\documentclass[11pt]{amsart}
\usepackage[top=1in, bottom=1in, left=1in, right=1in]{geometry}
\usepackage{amssymb}
\usepackage{amsfonts}
\usepackage{paralist}
\usepackage{tikz}
\usepackage{color}
\usepackage[all]{xy}
\usepackage[normalem]{ulem}
\usepackage[colorlinks,pagebackref=true,pdftex]{hyperref}
\usepackage[alphabetic,abbrev,lite]{amsrefs}

\makeatletter
\def\theenumi{\@roman\c@enumi}
\makeatother
\newcommand{\into}{\hookrightarrow}
\newcommand{\onto}{\twoheadrightarrow}
\newcommand{\xra}{\xrightarrow}

\theoremstyle{plain}
\newtheorem{theorem}{Theorem}[section]
\newtheorem{lemma}[theorem]{Lemma}

\newtheorem{prop}[theorem]{Proposition}
\newtheorem{cor}[theorem]{Corollary}

\newtheorem{claim*}[theorem]{Claim}
\newtheorem{thm}[theorem]{Theorem}

\theoremstyle{remark}

\newtheorem{example}[theorem]{Example}
\newtheorem{defn}[theorem]{Definition}

\numberwithin{equation}{section}

\def\<{\langle}
\def\>{\rangle}

\newcommand{\pos}{\operatorname{pos}}

\newcommand{\coker}{\operatorname{coker}}

\newcommand{\rank}{\operatorname{rank}}

\newcommand{\supp}{\operatorname{supp}}

\newcommand{\characteristic}{\operatorname{char}}

\newcommand{\kk}{\Bbbk}

\newcommand{\QQ}{\mathbb{Q}}
\newcommand{\VV}{\mathbb{V}}

\newcommand{\ZZ}{\mathbb{Z}}
\newcommand{\NN}{\mathbb{N}}
\newcommand{\WW}{\mathbb{W}}

\newcommand{\BQ}{{\operatorname{B_{\mathbb Q}}}}
\newcommand{\BQMCM}{{\operatorname{B_{\mathbb Q}^{\text{MCM}}}}}

\newcommand{\defi}[1]{{\upshape\sffamily #1}}

\newcommand{\excise}[1]{}

\newcommand{\e}{\epsilon}
\newcommand{\C}{C}
\renewcommand{\P}{P}

\title{The cone of Betti diagrams over a hypersurface ring of low embedding dimension}
\author[C. Berkesch]{Christine Berkesch}
\address{Department of Mathematics \\ Duke University, Box 90320 \\ Durham, NC 27708.}
\email{cberkesc@math.duke.edu}
\author[J. Burke]{Jesse Burke}
\address{Department of Mathematics\\ 
Universit\"at Bielefeld\\ 
33501 Bielefeld\\ 
Germany.}
\email{jburke@math.uni-bielefeld.de}
\author[D. Erman]{Daniel Erman}
\address{Department of Mathematics \\ University of Michigan \\Ann Arbor, MI 48109.}
\email{erman@umich.edu}
\author[C. Gibbons]{Courtney Gibbons}
\address{Department of Mathematics \\ University of Nebraska--Lincoln \\ Lincoln, NE 68588.}
\email{s-cgibbon5@math.unl.edu}
\thanks{CB was partially supported by NSF Grant OISE~0964985. 
DE was partially supported by NSF Award No. 1003997 and by a Simons Foundation Grant.}
\begin{document}
\vspace*{-.65cm}

\begin{abstract}
We give a complete description of the cone of Betti
diagrams over a standard graded hypersurface ring of the form
$\kk[x,y]/\<q\>$, where $q$ is a homogeneous quadric. 
We also provide a finite algorithm for decomposing
Betti diagrams, including diagrams of infinite projective dimension,
into pure diagrams. 
Boij--S\"oderberg theory completely describes the cone of Betti diagrams over a 
standard graded polynomial ring; our result provides the first example 
of another graded ring for which the cone of Betti diagrams is entirely understood.
\end{abstract}

\maketitle
\section{Introduction}
\label{sec:intro}

In an important shift in perspective to the study of graded free resolutions, 
Boij and S\"oderberg suggested that their numerics are easier to understand 
``up to scalar multiplication''~\cite{boij-sod1}. More specifically, for minimal free 
resolutions of graded modules over the standard graded  polynomial ring, they 
formulated precise conjectures about the possible Betti diagrams of such modules, 
including a description of the extremal rays of the cone of all such Betti diagrams.  
The subsequent proof of their 
conjectures~\cites{efw,eis-schrey1,boij-sod2,eis-schrey2} provided a 
breakthrough in our understanding of the structure of graded free resolutions, including
a proof of the Herzog--Huneke--Srinivasan Multiplicity Conjectures~\cite[Conjectures 1 and 2]{herzog-srinivasan}.

In this paper we investigate Boij--S\"oderberg theory for graded hypersurface 
rings, where the existence of resolutions of infinite projective dimension 
complicates the picture.  
Our main result is a complete description of the cone of Betti diagrams 
over a standard graded hypersurface ring of the form 
$\kk[x,y]/\<q\>$, where $q$ is a homogeneous quadric.
As in the case of a standard graded polynomial ring, 
there is a partial order on the extremal rays of the cone which gives it 
the structure of a simplicial fan. 
We obtain a similar result for standard graded rings 
of the form $\kk[x]/\<x^n\>$ for any $n\geq 2$. 
Although there has been recent work on extending Boij--S\"oderberg 
theoretic results to rings other than 
the polynomial ring~\cites{boij-floy,floystad,otherMRC,BEKS-local}, 
the main result of this paper provides the first example of another graded ring 
for which the cone of Betti diagrams is entirely understood.

For a standard graded Noetherian commutative ring $R$ and a finitely 
generated graded $R$-module $M$ with minimal graded free resolution
$F_\bullet\colon F_0 \leftarrow F_1 \leftarrow \cdots$, 
let $F_i=\bigoplus_{j\in \mathbb Z} R(-j)^{\beta^R_{ij}(M)}$.  Set $\VV$
to be the space of column finite matrices with entries in $\QQ$,
and define the \defi{Betti diagram} of $M$, denoted $\beta^R(M)$, 
to be the matrix whose entries are given by
\[ 
\left(\beta^R(M)\right)_{i,j} := \beta^R_{i,j}(M) \in \VV.
\]
We adopt the standard Betti diagram convention 
when displaying matrices in $\VV$, writing 
\[
\beta^R(M) = \begin{pmatrix}
    \vdots & \vdots & \vdots & \\
    ^*\beta^R_{0,0}(M) & \beta^R_{1,1}(M) &  \beta^R_{2,2}(M) & \cdots
    \\
    \beta^R_{0,1}(M) & \beta^R_{1,2}(M) &  \beta^R_{2,3}(M) & \cdots  \\
    \vdots &\vdots & \vdots &
  \end{pmatrix},
  \]
where the symbol $^*$ identifies the $(0,0)$-entry (this symbol may be omitted when the indexing is clear from context).

We define the \defi{cone of Betti diagrams over $R$} to be
\[
\BQ(R) := \left \{ \sum_{M} a_M \beta^R(M) \, \Bigg\vert \, a_M \in \QQ_{\geq
    0} \text{ and  almost all $a_M$ are zero} \right \} \subseteq \VV,
\]
so that it is the positive hull of $\beta^R(M)$ for all finitely
generated graded $R$-modules $M$.

As in the case of a polynomial ring, our description of the cone of Betti diagrams 
for the hypersurfaces above depends on the notion of a \defi{pure resolution}; this is a resolution 
of the form
\[
R(-d_0)^{\beta_{0,d_0}}\gets R(-d_1)^{\beta_{1,d_1}} 
\gets  R(-d_2)^{\beta_{2,d_2}}\gets  \cdots
\]
for some integers $d_0 < d_1 < d_2 < \cdots$ with $d_i \in \ZZ
\cup {\infty}$. (By convention $R(-\infty) = 0$ and $\infty <\infty$.)
We refer to $(d_0, d_1, d_2, \ldots)$ as the \defi{degree sequence} of
the pure resolution.

The simplest hypersurface ring $R$ is one of embedding dimension 1. 
The extremal ray description of this cone, 
provided in Proposition~\ref{prop:artinian}, follows from the structure theorem of 
finitely generated modules over a principal ideal domain. We give an
equivalent description of this cone in terms of facets in Theorem \ref{thm:main artinian}.

\begin{prop}\label{prop:artinian}
Let $R=\kk[x]/\<x^n\>$. The extremal rays of $\BQ(R)$ are the rays in 
$\VV$ spanned by:
\begin{enumerate}
	\item\label{thm:artinian:fin}
	the Betti diagrams of those modules of finite projective dimension having 
	a pure resolution of the form $(d_0, \infty, \infty, \ldots)$;
	\item\label{thm:artinian:inf} 
	the Betti diagrams of those modules of infinite projective dimension having 
	a pure resolution of type $(d_0,d_1,d_0+n,d_1+n,\ldots)$.
\end{enumerate}
\end{prop}


Our main result is a complete description of the cone 
$\BQ(R)$ when $R$ is a quadric hypersurface ring of embedding dimension $2$. 
We state here its description in terms of extremal rays; see Theorem
\ref{thm:main} for a description in terms of facets.

\begin{thm}\label{thm:main quadric}
Let $q$ be any quadric in $\kk[x,y]$, and let $R=\kk[x,y]/\<q\>$.
The extremal rays of $\BQ(R)$ are the rays in $\VV$ spanned by:
\begin{enumerate}
	\item the Betti diagrams of those Cohen--Macaulay modules of finite projective
          dimension having a pure resolution of the form 
          $(d_0, d_1, \infty, \ldots)$;
	\item the Betti diagrams of those finite length modules of infinite projective 
	  dimension having a pure resolution of type 
	  $(d_0,d_1,d_1+1,d_1+2,\ldots)$.
\end{enumerate}
\end{thm}

As in the main results of
Boij--S\"oderberg theory for the standard graded polynomial ring~\cite{eis-schrey1}, 
for both types of hypersurfaces $R$ above, our results provide
a simplicial fan structure on $\BQ(R)$
(for the definition of simplicial fan and other notions from convex geometry, 
see Appendix \ref{sec:notation}).

\begin{thm}\label{thm:simplicial fan}
Let $R$ be a standard graded hypersurface ring of the form 
$\kk[x]/\<x^n\>$ for any $n\geq 2$ or $\kk[x,y]/\<q\>$, where $q$ is any homogeneous quadric. 
Then the cone of Betti diagrams $\BQ(R)$ has the structure of a simplicial fan induced by a partial order on its extremal rays. 
\end{thm}

From the simplicial fan structure, we obtain 
decomposition algorithms for $R$-Betti diagrams as in \cite{eis-schrey1,boij-sod2}, 
as well as $R$-analogues of the Multiplicity Conjectures (see \S\ref{sec:mult and decomp}). 


New phenomena arise in the hypersurface case 
that are not seen in the case of a standard graded polynomial ring.  To begin with,
some of the functionals used to provide a halfspace description
of $\BQ(R)$ have no analogue in the polynomial ring case.
One set of these functionals directly measures the nonminimality of the
\defi{standard resolution}. This resolution, introduced in
\cite[\S3]{shamash} (see also \cite[\S7]{eisenbud-ci}), builds a free $R$-resolution from a minimal
free $S$-resolution.  The resulting functionals thus directly reflect 
the passage from the polynomial to hypersurface case.

Another interesting difference comes from the simplicial structure on 
$\BQ(R)$.  Unlike the polynomial ring, we cannot simply use the termwise 
partial order on $R$-degree sequences. 
Instead, we introduce partial orders that take into
account the infinite resolutions that occur over a hypersurface ring,  
see Definitions~\ref{def:quadric:degseq} and~\ref{def:artinian:degseq}.

Finally, we observe that for hypersurface rings, 
it is no longer the case that every Cohen--Macaulay module 
with a pure resolution lies on an extremal ray. 
This already happens in the context of Theorem~\ref{thm:main quadric}. 
For instance, let $M$ be the maximal Cohen--Macaulay module 
$\<x\>\subseteq R := \kk[x,y]/\<x^2\>$. The Betti diagram of $M$ is not extremal, as it decomposes as
\begin{align*}
\beta^R(M) 
\ =\ 
\begin{pmatrix} 1&1&1&1&\cdots\\ 
-&-&-&-&\cdots \end{pmatrix}
\ =\ 
\frac{1}{2}
\begin{pmatrix}1&2&2&2&\cdots \\ 
-&-&-&-&\cdots \end{pmatrix}
+\frac{1}{2}
\begin{pmatrix}1&-&-&-&\dots\\ 
-&-&-&-&\cdots \end{pmatrix}.
\end{align*}

Consideration of more general hypersurfaces complicates this situation even 
further and suggests some of the challenges in expanding this theory to
  hypersurfaces of higher degree.
\begin{prop}
  Let $R=\kk[x_1, \dots, x_r]/\<f\>$ be any hypersurface ring with
  $r>1$ and $\deg(f)>2$.  Then $\beta^R(\kk)$ is an extremal
  ray in $\BQ(R)$ which is not pure.
\end{prop}

\begin{proof}
The Tate resolution of $\kk$, introduced in \cite{Tate}, is the
minimal free resolution of $\kk$. Using this resolution, since $\deg(f)>2$ and $r>1$, one
  checks that the second syzygy module $\Omega^2(\kk)$ has minimal
  generators in degrees $2$ and $\deg(f)-1$, and thus $\beta^R(\kk)$
  is not pure.  

We next claim that if $M$ is any module generated in degree $0$,
then $\beta_{1,1}^R(M)\leq r\cdot \beta_{0,0}^R(M)$ with equality
  if and only if $M$ is a direct sum of copies of $\kk$. 
To see this, we first set $a:=\beta_{0,0}(M)$.  The linear first syzygies of $M$ will be linearly independent in
  the $\kk$-vector space $R^{a}_1$; since $\dim R_1=r$, this space is $r\cdot a$-dimensional, 
  implying the inequality.
    Now, if $M\cong \kk^a$, then equality holds by the
  Tate resolution. Conversely, if $\beta_{1,1}^R(M) = r\cdot a$ then
each generator of $M$ is annihilated by $(x_1, \ldots, x_r)$, 
 and so $M\cong \kk^a$.

Finally, to see that $\beta^R(\kk)$ is extremal, suppose that
$\beta^R(\kk) = \sum a_i \beta^R(M_i)$ for $R$-modules $M_i$ and $a_i
\in \QQ_{\geq 0}$. This implies that $\sum a_i \beta_{0,0}^R(M_i) =
\beta_{0,0}^R(\kk) = 1$. Using this, and the claim above, we have
\[r = \beta_{1,1}^R(\kk) = \sum a_i \beta_{1,1}^R(M_i) \leq \sum r a_i
\beta_{0,0}^R(M_i) = r\]
and so $\sum a_i \beta_{1,1}^R(M_i) = \sum r a_i
\beta_{0,0}^R(M_i)$. Since $\beta_{1,1}^R(M_i) \leq r 
\beta_{0,0}^R(M_i)$ for all $i$, we must have equality, which by the claim implies
that each $M_i$ is a direct sum of copies of $\kk$.
\end{proof}

This paper is outlined as follows. 
In \S\ref{sec:quadric}, we consider the case of a quadric hypersurface ring with embedding dimension 2. \S\ref{sec:artinian} is dedicated to the case of embedding dimension 1. \S\ref{sec:mult and decomp} addresses Theorem~\ref{thm:simplicial fan} and the $R$-analogues of the Multiplicity Conjectures. Finally, we include Appendix~\ref{sec:notation} on convex geometry. 

\subsection*{Acknowledgements}
We would like to thank the AMS and the organizers of 
the Mathematical Research Community on Commutative Algebra in June 2010, 
where this project began; 
we especially thank David Eisenbud for inspiring us to work on this problem. 
We are also grateful to Luchezar Avramov, W. Frank Moore, and Roger
Wiegand for helpful conversations.
Our work was aided by computations performed with \cite{M2}.

\section{Resolutions over hypersurface rings of embedding dimension $2$ and degree $2$}
\label{sec:quadric}

Set $S := \kk[x,y]$ and $R := S/\<q\>$ for a quadric $q$ in $S$. 
In this section, we give a full description of the cone of Betti diagrams of
$R$-modules, including a proof of Theorem~\ref{thm:main quadric}.

\begin{defn}\label{def:quadric:degseq}
We say that 
\[
d=(d_0,d_1,d_2,\dots)\in \prod_{i\in\NN} (\ZZ\cup\{\infty\})
\] 
is an $R$-\defi{degree sequence} if it has the form
\begin{enumerate}
\item $d = (d_0, \infty, \infty, \infty, \ldots),$ 
\item $d = (d_0, d_1, \infty, \infty, \ldots)$ with $d_0 < d_1$, or
\item $d = (d_0, d_1, d_1+1, d_1 + 2, \ldots)$ with $d_0 < d_1$. 
\end{enumerate}
We define a partial order $\leq$ on $R$-degree sequences as follows.  
We do a termwise comparison on the first two entries; in the case of a tie, 
we then do a termwise comparison on the remaining entries.  In other words, 
for two $R$-degree sequences $d,d'$ we say that $d\leq d'$ if either
\begin{itemize}
	\item  $d_0\leq d_0'$ and $d_1\leq d_1'$, with one of these inequalities 
		being strict, or
	\item  $d_0=d_0'$, $d_1=d_1'$, and $d_n\leq d_n'$ for all $n \geq 2$.
\end{itemize}
\end{defn}

Definition~\ref{def:quadric:degseq} leads to a decomposition 
algorithm (see \S\ref{sec:mult and decomp}) and fits into the 
framework of~\cite{BEKS-poset}.

Recall that the $\QQ$-vector space $\VV$ is the set of column-finite 
matrices with columns indexed by $i \in \NN$ and rows indexed by $j \in \ZZ$.
For each $R$-degree sequence $d$, 
we define a matrix $\pi_d\in \VV$ as follows.  Set $(\pi_d)_{0,j} = 1$ for $j = d_0$ and $0$ otherwise.
\begin{enumerate}
\item If $d = (d_0, \infty,  \ldots)$, set $(\pi_d)_{i,j} = 0$ for all $i \geq 1$ and all $j$.
\item If $d = (d_0, d_1, \infty, \ldots)$, set $\pi_{1,j} = 1$ when $j = d_1$ and $0$ otherwise.
\item If $d=(d_0,d_1,d_1+1,d_1+2,\dots)$, set $\pi_{i,j} = 2$ when $i \geq 1$ and $j = d_i$, and $0$ otherwise.
\end{enumerate}

\begin{example} 
Three degree sequences and their corresponding Betti diagrams appear below.
\begin{align*} 
\pi_{(0,\infty,\dots)} &=
	\begin{pmatrix}
		\vdots & \vdots & \vdots & \\
	   - & - & - & \cdots \\
       ^*1 & - & - & \cdots \\
       - & - & - & \cdots \\
       - & - & - & \cdots \\
       \vdots & \vdots & \vdots &
\end{pmatrix} \ \ \ 
\pi_{(1,2,\infty,\dots)} =
	\begin{pmatrix}
	   \vdots & \vdots & \vdots & \\
	    - & - & - & \cdots \\
       ^* - & - & - & \cdots \\
       1 & 1 & - & \cdots \\
       - & - & - & \cdots \\
       \vdots & \vdots & \vdots & 
       \end{pmatrix} \ \ \ 
 \pi_{(0,3,4,5,\dots)} = 
 	\begin{pmatrix}
 	   \vdots & \vdots & \vdots & \\
 	   - & - & - & \cdots \\
       ^*1 & - & - & \cdots \\
       - & - & - & \cdots \\
       - & 2 & 2 & \cdots \\
       \vdots & \vdots & \vdots &
	\end{pmatrix}
\end{align*}
\end{example}

We define functionals on $v\in\VV$ as follows: 
\[
\e_{i,j}(v) := v_{i,j}, \quad 
\alpha_{k}(v):=\e_{1,k}(v)-\e_{2,k+1}(v),
\quad \text{and} \quad
\gamma_k(v):=\sum_{j\leq k} \left( 
2\e_{0,j}(v)-2\e_{1,j+1}(v)+\e_{2,j+2}(v)
\right).
\]
Observe that the functional $\gamma_k$ is well-defined for any $v \in \VV$ 
because $v$ is column-finite.

\begin{example}
 The functional $\gamma_2$ applied to a Betti diagram $\beta^R(M)$ is
 given by taking the dot product of $\beta^R(M)$ with the following diagram: 
\[ 	
\begin{pmatrix}
           \vdots & \vdots & \vdots & \vdots & \\
           ^*2 & -2 & 1 & 0 &\cdots \\
      	   2 & -2 & 1 & 0 &\cdots \\
      	   2 & -2 & 1 & 0 &\cdots \\
            0 & \phantom{-}0 & 0 & 0 &\cdots \\
      	   \vdots & \vdots & \vdots & \vdots &
\end{pmatrix}. 
\]
\end{example}

\begin{thm}\label{thm:main}
The following cones in $\VV$ are equal:
\begin{enumerate}
	\item  The cone $\BQ(R)$ spanned by the Betti diagrams of all finitely generated $R$-modules.
	\item  The cone $D$ spanned by $\pi_d$ for all $R$-degree sequences $d$.
	\item  The cone $F$ defined to be the intersection of the halfspaces 
		\begin{enumerate}
			\item\label{eps}  
				$\{\e_{i,j} \geq 0\}$ for all $i\geq 0$ and $j = 0$ or $2$;
			\item\label{alphone}  
				$\{\alpha_{k}\geq 0\}$ for all $k\in \ZZ$;
			\item\label{gam}  
				$\{\gamma_k\geq 0\}$ for all $k\in \ZZ$; 
			\item\label{periodic2}
				$\{\pm(\e_{i,j} -\e_{i+1,j+1}) \geq 0\}$ for $i\geq 2, j\in \ZZ$.
		\end{enumerate}
\end{enumerate}
\end{thm}
To prove Theorem~\ref{thm:main}, we show the inclusions $D\subseteq
\BQ(R)\subseteq F \subseteq D$ which are contained in Lemmas~\ref{lem:DBQ(R)}, \ref{lem:BQ(R)F}, and \ref{lem:FD}, respectively.
The proof of Lemma~~\ref{lem:DBQ(R)} is straightforward, and the
proof of Lemma~\ref{lem:FD} largely
follows Boij and S\"oderberg's techniques involving convex polyhedral geometry.
By contrast, the proof of Lemma ~\ref{lem:BQ(R)F} requires new
ideas. In particular, we use a construction due to \cite{shamash} that
constructs a (not necessarily minimal) $R$-free resolution from an
$S$-free resolution; see also \cite[\S7]{eisenbud-ci}. We briefly recall
this construction now.

Let $G_\bullet$  be a graded free $S$-resolution of an $R$-module $M$ 
(recall that $S = \kk[x,y]$) 
Since multiplication by $q$ is nullhomotopic on $G_\bullet$, 
there are homotopy maps $s_1, s_2$:
\[
\xymatrix{ 0 & \ar[l] \ar[d]_q \ar[dr]^{s_1} G_0(-2) & \ar[l]
  \ar[d]_q \ar[dr]^{s_2} G_1(-2) &
  \ar[l] \ar[d]_q G_2(-2) & \ar[l] 0 \\
0 & \ar[l] G_0 & \ar[l] G_1 & \ar[l] G_2 & \ar[l] 0.
}
\]
Now set $\overline G_i := G_i \otimes R$, $\overline \partial_i = \partial_i \otimes R$, and $\overline s_i := s_i \otimes R$ for $i = 1,2$.  The resulting complex
\[
\xymatrix@C=3pc@R=3pc{
0 & \ar[l] \overline G_0 
& \ar[l]_{\ \overline \partial_1} \overline G_1 
& \ar[l]_{\begin{tiny}\begin{pmatrix} \overline \partial_2 \, , \, \overline s_1 \end{pmatrix}\end{tiny}\ \ \ }
{\begin{matrix} \overline G_2\\ \oplus \\ \overline G_0(-2)\end{matrix}}
& \ar[l]_{\ \begin{tiny}\begin{pmatrix} \overline s_2 \\ \overline \partial_1 \end{pmatrix}\end{tiny}} \overline{G}_1(-2) 
&\ar[l]_{\begin{tiny}\begin{pmatrix} \overline \partial_2 \, , \, \overline s_1 \end{pmatrix}\end{tiny}}  {\begin{matrix} \overline G_2(-2)\\ \oplus\\ \overline G_0(-4) \end{matrix}}
&\ar[l] \cdots
}
\]
is an $R$-free resolution of $M$.  Note that there are additional maps
$G_i \to G_{i+2d-1}$ in the construction given in
\cite[\S3]{shamash}. These maps are $0$ in our context because $G_i = 0$ when $i \geq 3$.

\begin{lemma}\label{lem:DBQ(R)}
There is an inclusion $D\subseteq \BQ(R)$.
\end{lemma}
\begin{proof}
It suffices to show that, for each $R$-degree sequence $d$, there exists an 
$R$-module $M$ with $\beta^R(M)=\pi_d$.  
If $d=(d_0,\infty,\dots)$, we simply choose $M=R(-d_0)$. For the other
cases, fix $\ell$, a linear form not a scalar multiple of $x$, that is a nonzero divisor on $R$ (i.e., $\ell$ does not divide $q$).  
Such an $\ell$ exists in any characteristic.
If $d=(d_0,d_1,\infty,\dots)$, we set $M=R(-d_0)/\<\ell^{d_1-d_0}\>$.

Finally, if $d=(d_0,d_1,d_1+1,d_1+2,\dots)$, we set $M$ to be 
$R(-d_0)/\<\ell^{d_1-d_0},x\ell^{d_1-d_0-1}\>$.  To see that $M$ has
the desired Betti diagram, we first consider the minimal $S$-free
resolution. By hypothesis $q, \ell^{d_1 - d_0}, x \ell^{d_1 - d_0 - 1}$ are
a minimal set of generators in $S$. Applying the Hilbert-Burch theorem, see
e.g.  \cite[20.15]{eis-ca}, the $S$-free resolution of $M$ has the form: \[
\xymatrix{
0 &\ar[l] S(-d_0) & \ar[l]_{\partial_1} {\begin{matrix} S(-d_0-2) \\ \oplus \\ S(-d_1)^2 \end{matrix}} &\ar[l] S(-d_1-1)^2 & \ar[l] 0.
}
\]
where $\partial_1 = \left [ \begin{matrix} q & \ell^{d_1 - d_0} & x
    \ell^{d_1 - d_0 - 1} \end{matrix} \right ]$. 
We fix homotopies $s_1,
s_2$ for multiplication by $q$ on this resolution:
\[
\xymatrix{ 0 & \ar[l] \ar[d]_q \ar[dr]^{s_1} S(-d_0-2) & \ar[l]
  \ar[d]_q \ar[dr]^{s_2}{\begin{matrix} S(-d_0-4) \\ \oplus \\ S(-d_1-2)^2 \end{matrix}} &
  \ar[l] \ar[d]_q S(-d_1-3)^2 & \ar[l] 0 \\
0 &\ar[l] S(-d_0) & \ar[l]_{\partial_1} {\begin{matrix} S(-d_0-2) \\ \oplus \\ S(-d_1)^2 \end{matrix}} &\ar[l] S(-d_1-1)^2 & \ar[l] 0.
}
\]
Since $\ell$ does not divide $q$ we see that the component of $s_1$ that
maps $S(-d_0 - 2)$ to $S(-d_0 -2)$ must be 1. By degree
considerations, the maps $s_1$ and $s_2$ cannot contain any other unit
entries.

The standard resolution of $M$ is now given by 
\[
\xymatrix{
0 &\ar[l] R(-d_0) & \ar[l] {\begin{matrix} R(-d_0-2) \\ \oplus \\ R(-d_1)^2 \end{matrix}} 
&\ar[l] {\begin{matrix}R(-d_1-1)^2 \\ \oplus \\ R(-d_0 - 2) \end{matrix}} 
& \ar[l] {\begin{matrix} R(-d_0-4) \\ \oplus \\ R(-d_1-2)^2 \end{matrix}}
& \ar[l] \cdots.
}
\]
The maps $R(-d_0 - 2n) \leftarrow R(-d_0-2n)$ 
are the only nonminimal part of this resolution.  It follows that 
$M$ has a minimal free $R$-resolution of the form \[
\xymatrix{
0 & \ar[l] R(-d_0) & \ar[l] R(-d_1)^2 & \ar[l] R(-d_1 - 1)^2 & \ar[l] R(-d_1 - 2)^2 & \ar[l] \cdots
,}
\]
which yields the desired Betti diagram.
\end{proof}

\begin{lemma}\label{lem:BQ(R)F}
There is an inclusion $\BQ(R) \subseteq F$.
\end{lemma}
\begin{proof}
  Fix a finitely generated graded $R$-module $M$. 
  We must show that the inequalities defining $F$ are
  nonnegative on $\beta^R(M)$. 
  Certainly $\e_{i,j}(\beta^R(M))=\beta^R_{i,j}(M)\geq 0$ 
  for all $i$ and $j$, completing case~\eqref{eps}.  
  For case~\eqref{periodic2}, the
  minimal resolution of $M$ is given by a matrix factorization after
  at most two steps by~\cite[Theorem~4.1]{eisenbud-ci}.  
  By~\cite[Lemma~20.11]{eis-ca}, 
  $\Omega^R_2(M)$ has depth $2$ and is thus maximal Cohen--Macaulay. 
  After extending the base field to its algebraic closure (which does not affect Betti diagrams),   the homogeneous quadric $q$ is, without loss of generality,
  either $x^2$ or $xy$. The matrix factorizations of these quadrics
  over an algebraically closed field
  are classified (see \cite[Example 6.5 and p. 76]{Yoshino}). 
  Thus the resolution of $M$
  after at most 2 steps is given by one of the matrix factorizations
  above;  one easily checks for these that~\eqref{periodic2} hold. 

For case~\eqref{alphone}, we show that $\alpha_{k}\left(\beta^R(M)\right) \geq 0$ by showing that it measures the rank of a map.  Fix a minimal free $S$-resolution $G_\bullet$ of $M$ as above, and let $s_1$ and $s_2$ denote the homotopies occurring in the standard resolution of $M$ over $R$.  Let $\sigma_{i,j}$ be
  the composition of the maps
\[ 
\sigma_{i,j}\colon S(-j)^{\beta^S_{i - 1, j - 2}(M)} \into G_{i-1}(-2) \xra{s_i} G_i
  \onto S(-j)^{\beta_{i,j}^S(M)}. 
\]
With the chosen basis, the entries of $\sigma_{i,j}$ have degree 0, so $\sigma_{i,j}$ 
is a matrix of elements of $\kk$. We claim that
\[
\alpha_{k}(\beta^R(M)):= \beta^R_{1,k}(M)  -\beta^R_{2, k+1}(M) = \rank \sigma_{2,k}\geq 0.
\]

It follows from this construction that 
\begin{align*} 
\beta^R_{1,k}(M) &= \beta^S_{1,k}(M) - \rank \sigma_{1,k} \\
 \text{and}\quad \beta^R_{3,k+2}(M) &= \beta^S_{1,k}(M) - \rank \sigma_{1,k} - \rank \sigma_{2,k}. 
\end{align*}
Thus $\beta^R_{1,k}(M) - \beta^R_{3,k+2}(M) = \rank \sigma_{2,k}$. 
As noted above in the proof of ~\eqref{periodic2}, $\beta^R_{2,k+1}(M) = \beta^R_{3,k+2}(M)$, which yields the claim. 

Finally, for case~\eqref{gam}, or $\gamma_k$, 
let $\phi_1\colon F_1\to F_0$ be a minimal presentation of $M$ over $R$ and set 
\[
  F_0':=\bigoplus_{j\leq k} R(-j)^{\beta^R_{0,j}(M)} 
  \qquad \text{and}\qquad 
  F_1':=\bigoplus_{j\leq k+1} R(-j)^{\beta^R_{1,j}(M)}.
\]
There are natural split inclusions $F_0'\subseteq F_0$ and
$F_1'\subseteq F_1$.  In particular, $\phi_1$ induces a map $\phi_1'\colon
F_1'\to F_0'$. We set $N:=\coker(\phi_1')$, and note that $\phi_1'$ is a
minimal presentation of $N'$. 
As such, $\beta^R_{0,j}(M)=\beta^R_{0,j}(N)$ for all $j\leq k$, and
$\beta^R_{1,j}(M)=\beta^R_{1,j}(N)$ for all $j\leq k+1$.
In addition, we claim that $\beta^R_{2,j}(M)=\beta^R_{2,j}(N)$ for all $j\leq k+2$. 
To see this, consider the diagram
\[ 
\xymatrix{& & 0 \ar[d] & 0 \ar[d] & & \\ 
0 \ar[r] & \Omega_2^R(N) \ar[d] \ar[r] & F_1' \ar[d]
  \ar[r] & F_0' \ar[d]
  \ar[r] & N \ar[r] \ar[d] & 0\\
0 \ar[r] & \Omega_2^R(M) \ar[r] & F_1 \ar[r] & F_0
  \ar[r] & M \ar[r] & 0, }
\] 
where we view $\Omega_2^R(N)$, $\Omega_2^R(M)$ as submodules
 of $F_1', F_1$ respectively. By the Snake Lemma, $\Omega_2^R(N)$ is 
 a submodule of $\Omega_2^R(M)$. For a fixed  basis of $F_1$, any 
 element of $\Omega_2^R(M)$ may be
 written as a linear combination of the basis elements with
 coefficients in $R_{\geq 1}$. Thus for an element $x \in
 \Omega_2^R(M)$ of degree $j$ with $j \leq k +2$, we see that the
 basis elements whose corresponding coefficients are nonzero in a
 decomposition of $x$ have degree at most $j - 1 \leq k + 1$. 
 In particular, these basis elements are in $F_1'$, and hence  
 $\Omega_2^R(N)_j = \Omega_2^R(M)_j$ for all $j \leq k+2$, which implies the claim.

By the definition
of $\gamma_k$, we have now shown that
$\gamma_k\left(\beta^R(M)\right)=\gamma_k\left(\beta^R(N)\right)$. 
It thus suffices to show that $\gamma_k\left(\beta^R(N)\right)\geq 0$.  
We achieve this by showing that $\gamma_k\left(\beta^R(N)\right)=h_{k+2}(N)$, 
where $h_{k+2}(N)$ denotes the Hilbert function of $N$ in degree $k+2$.  

The Hilbert function of $N$ can be computed entirely in terms of
$\beta^R(N)$:
\begin{align*}
h_{k+2}(N) &=  \sum_{j\in \ZZ} \sum_{i=0}^\infty (-1)^i\beta^R_{i,j}(N) h_{k+2}(R(-j))\\
&= \sum_{j\in \ZZ} \sum_{i=0}^\infty (-1)^i\beta^R_{i,j}(N) h_{k+2-j}(R)\\
&=\sum_{\ell\in \ZZ} \sum_{i=0}^\infty (-1)^i\beta^R_{i,i+\ell}(N)h_{k+2-i-\ell}(R).
\intertext{Since $\beta^R_{0,j}(N)=0$ for $j>k$, $\beta^R_{1,j}(N)=0$ for $j>k+1$, 
and $h_i(R)=2$ for all $i>0$, we have that}
h_{k+2}(N) &=\sum_{\ell\leq k} \sum_{i=0}^\infty (-1)^i\beta^R_{i,i+\ell}(N)h_{k+2-i-\ell}(R)\\
&=\sum_{\ell\leq k} \left( \sum_{i=0}^{k+1-\ell} (-1)^i\beta^R_{i,i+\ell}(N)\cdot 2\right) + (-1)^
{k+2-\ell}\beta^R_{k+2-\ell,k+2}(N)\cdot 1.\\
\intertext{By applying ~\eqref{periodic2} twice, we see that
$\beta^R_{i,j}(N) = \beta^R_{i+2,j+2}(N)$ for all $i \geq 2$. 
Using this to cancel, we obtain}
h_{k+2}(N) &= \sum_{\ell \leq k} \left(2\beta^R_{0,\ell}(N)-2\beta^R_{1,\ell+1}(N)
+\beta^R_{2,\ell+2}(N)\right) \ = \gamma_k(\beta^R(N)).\qedhere
\end{align*}
\end{proof}

For the final inclusion in the proof of Theorem~\ref{thm:main}, 
we compare the cone $D$ (which is defined in terms of extremal rays) 
and the cone $F$ (which is defined in terms of halfspaces).  
As we see in Lemma~\ref{lem:boundaryD}, 
it is easier to move between these two descriptions in the case of a simplicial fan, 
so we first construct a simplicial fan $\Sigma$ whose support is contained in $D$.

\begin{lemma}\label{lem:simplicial}
For every finite chain $\C$ of $R$-degree sequences, the cone 
$\pos(\C):=\QQ_{\geq 0}\{\pi_d \mid d\in C\}$ is simplicial. 
The collection of these simplicial cones forms a simplicial fan.
\end{lemma}
\begin{proof}
The diagrams $\pi_d$ from any finite chain $\C$ are
linearly independent. This follows from the fact that for any degree sequence
$d$, $\pi_d$ has a nonzero entry in a position such
that, for every degree sequence $d'$ in the chain $C$ with $d < d'$, $\pi_{d'}$ has a zero in the corresponding position.

For the second statement, we need to show that these cones meet along faces.  
Using the observation above, the proof of \cite[2.9]{boij-sod1}
applies directly to our situation.
\end{proof}

\begin{lemma}\label{lem:FD}
There is an inclusion $F \subseteq D$.
\end{lemma}
\begin{proof}
Let $\Sigma$ be the simplicial fan constructed in Lemma~\ref{lem:simplicial}, 
and let $\supp(\Sigma)$ denote its support, as defined in Appendix~\ref{sec:notation}. 
By construction, $\supp(\Sigma) \subseteq D$, so it suffices to prove that $F \subseteq \supp(\Sigma)$.\footnote{A priori, $\supp(\Sigma)$ is a (not necessarily convex) subcone of $D$; the proof of Theorem~\ref{thm:main artinian} implies that $\supp(\Sigma)=D$.}
Now, we have a simplicial fan $\Sigma$ defined in terms of extremal rays, 
and we seek to determine its boundary halfspaces, as defined in
Appendix~\ref{sec:notation}. Then to prove the Lemma it will be
enough to show that each of the boundary halfspaces of $\Sigma$ is
contained in the list of halfspaces defining $F$.

In order to apply Lemma~\ref{lem:boundaryD}, we first reduce to the case of a 
full-dimensional, equidimensional simplicial fan in a finite dimensional vector space. For each $m \in \ZZ_{\geq 0}$, define the subspace $\VV_m$ of $\VV$ to be 
\[
\VV_m := \{ v \in \VV \mid v_{i,j} = 0 \text{ unless } -m+i \leq j \leq m+i\}.
\]
Note that $\VV_m$ contains the Betti diagram of any module with generators in degrees at least $-m$ and with regularity at most $m$.

Set $\Sigma_m:=\Sigma\cap \VV_m$ and $F_m := F \cap \VV_m$. 
Observe that 
$\Sigma_m = \{ \pos(\C) \mid {\C \text{ is a chain in } \P_m} \}$, 
where $\P_m := \{ \text{degree sequences } d \mid \pi_d \in \VV_m \}$, 
so by Lemma \ref{lem:simplicial}, $\Sigma_m$ is a simplicial fan.
Since $\VV = \bigcup_{m \geq 0} \VV_m$, it is enough to show that
$F_m \subseteq \supp(\Sigma_m)$ for all $m \geq 0$.

Next, we define the finite dimensional vector space 
\[
\overline{\VV}_m := \{ v \in \VV_m \mid v_{i,j} = 0 \text{ unless } i \leq 2\}, 
\]
and consider the projection $\Phi_m\colon \VV_m \to \overline{\VV}_m$. 
Since every pure diagram $\pi_d$ satisfies the functional of 
type~\eqref{periodic2} in the definition of $F$,
it follows that $\Phi_m$ induces an isomorphism of $\Sigma_m$
onto its image, which we denote by $\overline{\Sigma}_m$. There is
also an isomorphism of $F_m$ onto its image $\overline{F}_m$, since
the defining halfspaces of $F_m$ contain \eqref{periodic2}.
It thus suffices to show that $\overline{F}_m \subseteq \supp(\overline{\Sigma}_m)$ for all $m$.  

We claim that $\overline{\Sigma}_m$ is 
$(\dim \overline{\VV}_m)$-equidimensional.
Every maximal chain of degree sequences in $\P_m$ begins with 
$(-m,-m+1,m+n,\dots)$ and ends with $(m,\infty,\infty,\infty,\dots)$.  
For a fixed maximal chain $\C$, there is a unique $m' \leq m$ 
such that $\C$ is 
\[
(-m,-m+1,-m+n,\dots)<\dots< (m',m+1,m'+n,\dots)<(m',\infty,\infty,\dots)<\dots <(m,\infty,\infty,\dots).
\] 
From this observation, it follows that
\begin{align*} 
|\C| &= \left( (m+m)'+2(2m+1)\right) + \left( m-m'+1\right)=6m+3.
\end{align*}
Since the set $\{\pi_d\}$ is linearly independent for $d \in \C$ by
Lemma~\ref{lem:simplicial}, these diagrams form a basis
of $\overline{\VV}_m$.  It follows that $\overline{\Sigma}_m$ is a $(\dim \overline{\VV}_m)$-equidimensional
simplicial fan.

We now record a collection of supporting halfspaces which define $\overline{F}_m$: 
\begin{enumerate} 
\item\label{eps'}
	$\{\e_{i,j} \geq 0\}$ for all $i\geq 0$, $j\in \ZZ \cap [-m+i, m+i]$; 
\item\label{alphaone'} 
	$\{\alpha_{k}\geq 0\}$ for all $k\in \ZZ \cap [-m,m]$;  
\item\label{gam'}
	$\{\gamma_{k,m} 
	\geq 0\}$ for all $k\in \ZZ \cap [-m,m]$, where for $k \in \ZZ \cap [-m,m]$ we set 
	\[
	\gamma_{k,m} := \sum_{j = -m}^k \left(2\e_{0,j} -2\e_{1,j+1}+\e_{2,j+2} \right). 
	\]
\end{enumerate}
To complete the proof, we show that 
each boundary halfspace of $\overline{\Sigma}_m$
corresponds to a supporting halfspace of $\overline{F}_m.$
By Lemma~\ref{lem:boundaryD}, each boundary halfspace of 
$\overline{\Sigma}_m$ is determined by (at least one) boundary facet, 
and hence is determined by some
submaximal chain in the poset $\P_m$
that is uniquely extended to a maximal chain. 
The proof of \cite[Proposition~2.12]{boij-sod1} applies in our context, 
showing that each boundary halfspace of $\overline{\Sigma}_m$ depends on 
only a small part of any submaximal chain to which it corresponds.  
Namely, such a halfspace is determined by the unique $R$-degree sequence 
$d$ that extends a corresponding chain to a maximal one, along with its two 
neighbors $d' < d''$ in this extended chain, if they exist. 
We write this data as $\cdots <d'<d{\widehat{\,\,\,}}<d''<\cdots$.
By direct inspection of $\P_m$ (see Figure \ref{fig:bounded poset} for the case $m = 1$), 
the submaximal chains that can be uniquely extended are of the following forms:
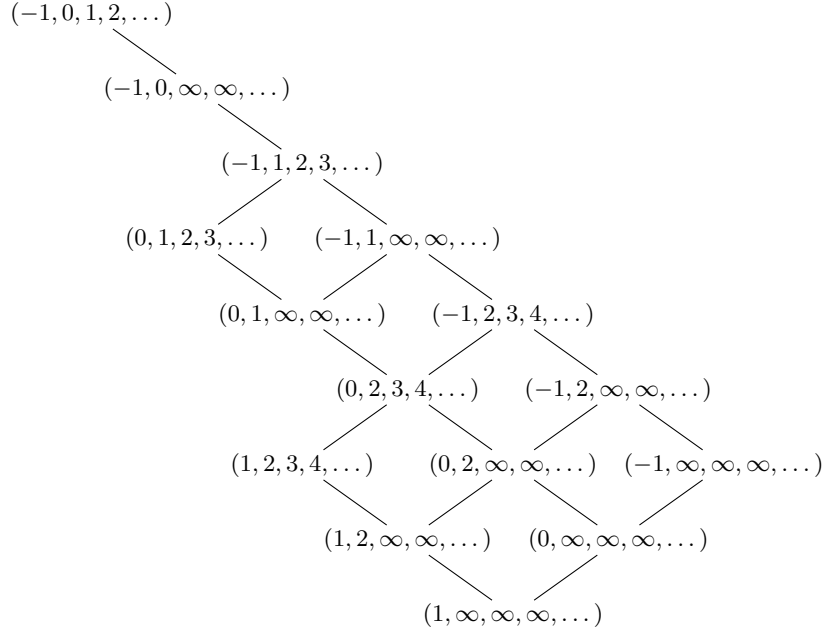
\begin{figure}
\begin{footnotesize}
\begin{tikzpicture}[xscale=1.4,yscale=1.0]
\draw(0,9) node {$(-1,0,1,2,\dots)$};
\draw(1,8) node {$(-1,0,\infty,\infty,\dots)$};
\draw(2,7) node {$(-1,1,2,3,\dots)$};
\draw(3,6) node {$(-1,1,\infty,\infty,\dots)$};
\draw(4,5) node {$(-1,2,3,4,\dots)$};
\draw(5,4) node {$(-1,2,\infty,\infty,\dots)$};
\draw(6,3) node {$(-1,\infty,\infty,\infty,\dots)$};

\draw(1,6) node {$(0,1,2,3,\dots)$};
\draw(2,5) node {$(0,1,\infty,\infty,\dots)$};
\draw(3,4) node {$(0,2,3,4,\dots)$};
\draw(4,3) node {$(0,2,\infty,\infty,\dots)$};
\draw(5,2) node {$(0,\infty,\infty,\infty,\dots)$};

\draw(2,3) node {$(1,2,3,4,\dots)$};
\draw(3,2) node {$(1,2,\infty,\infty,\dots)$};
\draw(4,1) node {$(1,\infty,\infty,\infty,\dots)$};

\draw[-] (0.2,8.8)--(0.8,8.2);
\draw[-] (1.2,7.8)--(1.8,7.2);
\draw[-] (2.2,6.8)--(2.8,6.2);
\draw[-] (3.2,5.8)--(3.8,5.2);
\draw[-] (4.2,4.8)--(4.8,4.2);
\draw[-] (5.2,3.8)--(5.8,3.2);

\draw[-] (1.2,5.8)--(1.8,5.2);
\draw[-] (2.2,4.8)--(2.8,4.2);
\draw[-] (3.2,3.8)--(3.8,3.2);
\draw[-] (4.2,2.8)--(4.8,2.2);

\draw[-] (2.2,2.8)--(2.8,2.2);
\draw[-] (3.2,1.8)--(3.8,1.2);

\draw[-] (1.8,6.8)--(1.2,6.2);
\draw[-] (2.8,5.8)--(2.2,5.2);
\draw[-] (3.8,4.8)--(3.2,4.2);
\draw[-] (4.8,3.8)--(4.2,3.2);
\draw[-] (5.8,2.8)--(5.2,2.2);

\draw[-] (2.8,3.8)--(2.2,3.2);
\draw[-] (3.8,2.8)--(3.2,2.2);
\draw[-] (4.8,1.8)--(4.2,1.2);
\end{tikzpicture}
\end{footnotesize}
\caption{The poset of degree sequences whose Betti diagrams
  lie in ${\VV}_1$.}
\label{fig:bounded poset}
\end{figure}
\begin{enumerate}[\upshape(a)]
	\item\label{type1} 
		$\cdots <d'<d{\widehat{\,\,\,}}<d''<\cdots$, where 
		$d'$ and $d''$ have projective dimension $1$ and either 
		$d''_0-d'_0=1$ or $d''_1-d'_1=1$  (but not both).  
		For instance, 
	\[
	\cdots < (0,1,\infty,\infty, \dots)<(0,2,3,4,\dots){\widehat{\,\,\,}}<
	(0,2,\infty,\infty, \dots)<\cdots.
	\]
	\item\label{type2} 
	$\cdots <d'<d{\widehat{\,\,\,}}<d''<\cdots$, where 
		$d'$ and $d''$ have infinite projective dimension and either 
		$d''_0-d'_0=1$ or $d''_1-d'_1=1$ (but not both). 
		For instance, 
	\[
	\cdots < (-1,1,2,3,\dots) <(-1,1,\infty,\infty, \dots){\widehat{\,\,\,}}<
	(-1,2,3,4, \dots)<\cdots.
	\]
	\item\label{type3} 
		$\cdots< (d_0', d_0'+1, \infty, \infty, \dots)<d{\widehat{\,\,\,}}<
		(d_0'+1, d_0'+2, d_0'+3, d_0'+4, \dots)<\cdots$. 
		For instance, 
	\[
	\cdots < (0,1,\infty,\infty,\dots) <(0,2,3,4, \dots){\widehat{\,\,\,}}<
	(1,2,3,4, \dots)<\cdots.
	\]
	\item\label{type4} 
	$\cdots <d'<d{\widehat{\,\,\,}}<d''<\cdots$, where 
		$d'$ and $d''$ differ by two in the first entry. 
		For instance, 
	\[
	\cdots < (-1,2,3,4,\dots) < (0,2,3,4,\dots){\widehat{\,\,\,}} < 
	(1,2,3,4) < \cdots \text{ or}
	\]
	\[
	\cdots < (1,\infty,\infty,\infty,\dots) <
	(2,\infty,\infty,\infty, \dots){\widehat{\,\,\,}}<
	(3,\infty,\infty,\infty,\dots)<\cdots.
	\]
	\item\label{christine} $\cdots< (d_0,m+1,m+2,m+3,\dots) < (d_0,m+1,\infty,\infty,\dots){\widehat{\,\,\,}} < (d_0,\infty,\infty,\infty)<\cdots$, for instance
	\[
	\cdots <(0,2,3,4)<(0,2,\infty,\infty){\widehat{\,\,\,}} < (0,\infty,\infty,\infty)<\cdots
	\]
	\item\label{type03}  
		$\cdots < d' < d{\widehat{\,\,\,}}$, where $d' = (m,m+1,\infty,\infty,\dots)$ 
		and $d = (m,\infty,\infty,\infty,\dots)$ 
		is the maximal element of its chain.
	\item\label{type04}  
		$\cdots < d' < d{\widehat{\,\,\,}}$, where $d' = (m-1,\infty,\infty,\infty,\dots)$ 
		and $d = (m,\infty,\infty,\infty,\dots)$ 
		is the maximal element of its chain.
	\item\label{type01}  
		$d{\widehat{\,\,\,}} < d'' <\cdots$, 
		where $d$ is the minimal element of its chain.
\end{enumerate}

We can show on a case by case basis 
that each boundary halfspace of $\overline{\Sigma}_m$ 
(as determined by a submaximal chain from the list above) 
corresponds to one of the halfspaces defining $\overline{F}_m$; 
we provide details for a portion of case~\eqref{type1}.  
Consider a submaximal chain $\C$ of the form 
\[
\cdots <(d_0,d_1-1,\infty,\dots)<(d_0,d_1,d_2,\dots){\widehat{\,\,\,}}<(d_0,d_1,\infty,\dots)<\cdots,
\] 
where $d_2 = d_1+1$. 
Note that $\e^*_{2,d_2}(\pi_c) = 0$ for all $c = (c_0, c_1, \ldots) \in \C$ 
because either $c_2 < d_2$ or $c_2 > d_2$. 
This shows that $\pi_c$ lies in the hyperplane $\{\e_{2,d_2}^*=0\}$ 
for all $c\in C$.  Since, in addition, 
$\e^*_{2,d_2}(\pi_{d_0,d_1,d_2,\dots}) = 1$, it follows that 
$C$ corresponds to the halfspace $\{\e_{0,d_0+1}^*\geq 0\}$.  

Using similar arguments, we see that a submaximal chain of 
type~\eqref{type1} or~\eqref{type01} corresponds to $\{\e_{2,d_2}^* \geq 0\}$; 
type~\eqref{type2} corresponds to $\{\alpha_{d_1} \geq 0\}$ and type~\eqref{christine}
corresponds to $\{\alpha_{m+1}\geq 0\}$; 
type~\eqref{type3} or~\eqref{type03} to $\{\gamma_{d_0',m} \geq 0 \}$; 
and finally, chains of types~\eqref{type4} and~\eqref{type04} correspond to $\{\e_{0,m}^*\geq 0\}$. 
\end{proof}

\begin{proof}[Proof of Theorem~\ref{thm:main quadric}]
Let $E$ be the cone spanned by Betti diagrams of
  extremal modules of finite projective dimension and extremal modules
  of infinite projective dimension with the stated degree sequences.  
  We see that $D \subseteq E$ by Lemma~\ref{lem:DBQ(R)}, 
  noting that the modules there are extremal of finite
  projective dimension or of infinite projective dimension with the
  correct degree sequence. Thus by Theorem~\ref{thm:main}, 
  we have $\BQ(R) = D$, as desired.
\end{proof}

\medskip
\section{Resolutions over hypersurface rings of embedding dimension $1$  and degree at least $2$}
\label{sec:artinian}

\smallskip
Set $S := \kk[x]$ and $R := S/\<x^n\>$ for some $n \geq 2$. 
In this section, we give a full description of the cone of Betti diagrams of
$R$-modules, as well as its implications for the cone of Betti diagrams of maximal 
Cohen--Macaulay modules over any standard graded hypersurface ring.

\medskip
\begin{defn}\label{def:artinian:degseq}
We say that 
\[
d=(d_0,d_1,\dots)\in \prod_{i\in\NN} (\ZZ\cup \{\infty\})
\] 
is an $R$-\defi{degree sequence} if it has the form
\begin{enumerate}
\item $d = (d_0, \infty, \infty, \infty, \ldots)$ or
\item $d = (d_0, d_1, d_2,\ldots)$, 
	where $d_0 < d_1$ and $d_{i+2} - d_i = n$ for all $i \geq 0$.
\end{enumerate}
We define a partial order on $R$-degree sequences as follows:  
if $d$ has finite projective dimension and $d'$ has infinite dimension, then
we say that $d<d'$; otherwise, we use the termwise partial order.
\end{defn}

\medskip
Given an $R$-degree sequence $d = (d_0, d_1, \dots)$, 
we define a diagram $\pi_d \in \VV$ by 
\begin{equation*}
(\pi_d)_{i,j} = \begin{cases} 1 &\quad \text{if } j = d_i \not = \infty, \\ 
				0 &\quad \text{otherwise}.
				\end{cases}
\end{equation*}

\medskip
\begin{example}
If $n = 3$, then 
\[ 
\pi_{(0,\infty,\infty,\infty,\dots)}
  = \begin{pmatrix} 
  \vdots & \vdots & \\
  - & - & \cdots \\
  ^*1 & - & \cdots \\ 
  - & - & \cdots \\ 
  - & - & \cdots \\
  \vdots & \vdots & \end{pmatrix} 
  \quad \text{and} \quad \pi_{(0,1,3,4,\dots)}
  = \begin{pmatrix}
  	\vdots & \vdots & \vdots & \vdots & \vdots & \\
    - & - & - & - & - & \cdots \\
    ^*1 & 1 & - & - & - & \cdots  \\
    - & - & 1 & 1 & - & \cdots \\
    - & - & - & - & 1 & \cdots \\
    - & - & - & - & - & \cdots \\
    \vdots & \vdots & \vdots & \vdots & \vdots & 
  \end{pmatrix}.
\] 
\end{example}

\medskip
To describe the cone $\BQ(R)$, we define the following functionals on $v\in\VV$: 
\[
\e_{i,j}(v) := v_{i,j}, 
\qquad
\alpha_{i,k}(v) := \e_{i,k}(v) - \e_{i+2,k+n}(v), 
\qquad
\theta_k(v) := \sum_{j\leq k} \e_{2,j}(v) - \sum_{j\leq k-n+1} \e_{1,j}(v),\]
\[
\text{and} \qquad \eta_k(v) := \sum_{j \leq k} \left(\e_{1,j}(v) - \e_{2,j+1}(v)\right).
\]

\medskip
\begin{example}
The functional $\eta_3$ 
applied to a Betti diagram $\beta^R(M)$ is
given by taking the dot product of $\beta^R(M)$ with the following diagram:
\[ 
\left ( \begin{array}{rrrrrr}
       \vdots & \vdots & \vdots & \vdots  &\\
       ^*0& 1 & -1 & 0 & 0 &\cdots\\
       0&  1 & -1 & 0 & 0&\cdots \\
       0& 1 & -1 & 0 & 0&\cdots \\
       0& 0 & 0 & 0 & 0 & \cdots\\
       \vdots & \vdots & \vdots & \vdots& \\
       \end{array} \right ). 
\]
\end{example}

\begin{thm}\label{thm:main artinian}
The following cones in $\VV$ are equal:
\begin{enumerate}
	\item  The cone $\BQ(R)$ spanned by the Betti diagrams 
		of all finitely generated $R$-modules.
	\item  The cone $D$ spanned by $\pi_d$ for all $R$-degree sequences $d$.
	\item\label{item:F}  
		The cone $F$ defined as the intersection of the halfspaces 
		\begin{enumerate}
			\item\label{eps artinian} 
				$\{\e_{i,j} \geq 0\}$ for $i = 0, 1, 2$ and $j \in \ZZ$;
			\item\label{alpha artinian}
				$\{\alpha_{0,k} \geq 0\}$ for all $k \in \ZZ$;
			\item\label{theta artinian} 
				$\{\theta_k \geq 0\}$ for all $k \in \ZZ$;
			\item\label{eta artinian} 
				$\{\eta_k \geq 0\}$ for all $k \in \ZZ$;
			\item\label{periodic}
				$\{\pm \alpha_{i,k}
				\geq 0\}$ for all $i\geq 1, k\in \ZZ$;
			\item\label{beta12}
				$\{ \pm \eta_\infty
				\geq 0 \}$. 
		\end{enumerate}
\end{enumerate}
\end{thm}
\begin{proof}
The equality $\BQ(R)=D$ follows from the structure theorem 
of finitely generated modules over principal ideal domains.  
Using extremal rays, it is also straightforward to check that 
$\BQ(R)\subseteq F$. 
We complete the proof by showing that $F\subseteq D$. 

For this final inclusion, note that the proof of Lemma~\ref{lem:simplicial} 
also holds in this context, so that 
$\Sigma = \{\pos(\C)\mid\text{$\C$ is a finite chain of $R$-degree sequences}\}$ 
is a simplicial fan; it suffices to prove that $F \subseteq \supp(\Sigma)$. 
Let $\overline{\VV}$ denote the natural projection of $\VV$ that sends   
$v\in\VV$ to its first three columns, denoted $v\mapsto \overline{v}$. 
Denote the respective images of $F$ and $\Sigma$ under this map by 
$\overline{F}$ and $\overline{\Sigma}$. 
For $m\geq 0$, let $P_m := 
\{\text{degree sequences } d\mid \overline{\pi}_d\in\overline{\VV}_m\}$, 
and define $\overline{\VV}_m$, $\overline{\Sigma}_m$, and $\overline{F}_m$ in a manner analogous to the proof of Lemma~\ref{lem:FD}. 
%
Since every pure diagram $\pi_d$ satisfies the functionals of 
types~\eqref{periodic} and~\eqref{beta12} in the definition of $F$,
it now suffices to show that 
$\overline{F}_m \subseteq \supp(\overline{\Sigma}_m)$ for all $m\geq 0$.  
%
Note that $\overline{F}_m$ and $\supp(\overline{\Sigma}_m)$ are both contained in the subspace $\overline{\WW}_m$ of $\overline{\VV}_m$ 
given by 
\[ 
\overline{\WW}_m := 
\left\{
\overline{v}\in\overline{\VV}_m \ \big\vert\  
\eta_\infty(\overline{v}) = 0 \ \text{ and }\ 
\overline{v}_{2,j} = 0 \text{ for } -m+2 \leq j < n-m.
\right\}.
\] 
We thus view them as objects in there. 

Direct computation shows that $\overline{\Sigma}_m\subseteq \overline{\WW}_m$ is a full-dimensional, equidimensional simplicial fan. 
To work towards $\overline{F}_m \subseteq \supp(\overline{\Sigma}_m)$, note that  
defining halfspaces for $\overline{F}_m\subseteq \overline{\WW}_m$ are:
\begin{align*}
\{ \e_{i,j} \geq & \; 0 \}
\quad \text{for } i \in \{0,1,2\} \text{ and } j\in \ZZ \cap [-m+i, m+i], \\
\{\alpha_{0,j} \geq & \; 0 \}
\quad \text{for } j\in \ZZ\cap[-m, m+2-n], \\
\bigg\{ \theta_{k,m} 
	:= 
	\sum_{j = -m+2}^k \e_{2,j}
	- \sum_{j=-m+1}^{k-n+1} \e_{1,j} \geq & 0 
\bigg\}
\quad \text{for } k\in \ZZ\cap [n-m-1,m+2], \quad \text{and} \\
\bigg\{
{\eta_k}_m :=
	\sum_{j = -m+1}^k  (\e_{1,j} - \e_{2,j+1}) \geq & 0 
\bigg\}
\quad \text{for } k\in \ZZ\cap [-m+1,m+1].
\end{align*}
Each boundary halfspace of $\overline{\Sigma}_m$ depends on 
certain submaximal chains given by data of the form 
$\cdots<d'<d{\widehat{\,\,\,}}<d''<\cdots$.
Such submaximal chains take the following forms:
\begin{enumerate}[\upshape(a)]
\item\label{typeR} 
	$\cdots<(d_0, d_1, \dots) < (d_0 + 1, d_1,\dots){\widehat{\,\,\,}} < 
	(d_0 + 2, d_1,\dots)<\cdots$, where $d_1 < \infty$;
\item\label{typeL} 
	$\cdots<(d_0, d_1, \ldots) < (d_0, d_1 + 1, \ldots){\widehat{\,\,\,}} < 
	(d_0, d_1 + 2, \ldots)<\cdots$, where $d_1 < \infty$;
\item\label{typeE} 
	$\cdots<(d_0, d_0 + 1, \ldots) < (d_0, d_0 + 2, \ldots){\widehat{\,\,\,}} < 
	(d_0 + 1, d_0 + 2, \ldots)<\cdots$; 
\item\label{typeE'}
	$\cdots<(d_0, d_0 + n - 1, \ldots) < (d_0+1, d_0 + n - 1, \ldots){\widehat{\,\,\,}} < 
	(d_0 + 1, d_0 + n, \ldots)<\cdots$; 	
\item\label{typeE''}
	$\cdots<(m-n+2, m, \ldots) < (m-n+2, m+1, \ldots){\widehat{\,\,\,}} < 
	(-m, \infty, \ldots)<\cdots$; 	
\item\label{typeS}
	$\cdots<(m-n+2, m+1, \ldots) < (-m, \infty, \ldots){\widehat{\,\,\,}} < 
	(-m+1, \infty, \ldots)<\cdots$; 	
\item\label{typeI}
	$\cdots<(d_0, \infty, \ldots) < (d_0+1, \infty, \ldots){\widehat{\,\,\,}} < 
	(d_0+2, \infty, \ldots)<\cdots$; 	
\item\label{top} $(-m,-m+1,
					\dots){\widehat{\,\,\,}} < (-m,-m+2,\dots)<\cdots$;
\item\label{bottom} $\cdots<(m-1,\infty,\dots) < (m,\infty,\dots){\widehat{\,\,\,}}$.
\end{enumerate}
One may now verify that the boundary halfspaces corresponding to these submaximal chains are, respectively: 
\begin{enumerate}[\upshape(a)]
	\item $\{\e_{2,d_0+1+n}\geq 0\}$; 
	\item $\{\e_{1,d_1+1}\geq 0\}$;
	\item $\{ 
		\theta_{d_0+n,m}
		\geq 0\}$; 
	\item $\{ 
		{\eta_{d_0+n-1}}_m
		\geq 0\}$; 
	\item $\{ \e_{1,m+1} \geq 0\}$; 
	\item $\{
		\alpha_{0,-m}
		\geq 0\}$; 
	\item $\{ 
		\alpha_{0,d_0+1}
		\geq 0\}$; 
	\item $\{ \e_{1,-m+1} \geq 0\}$; 
	\item $\{ \e_{0,m} \geq 0\}$ if $n>2$ or $\{\alpha_{0,m}\geq 0\}$ if $n=2$. 
\end{enumerate}
As each of these halfspaces appear in our definition of $\overline{F}_m$ above, 
we obtain $F\subseteq D$, as desired. 
\end{proof}

As illustrated by the following corollary, Theorem~\ref{thm:main artinian} 
has implications for the study of Betti diagrams of maximal 
Cohen--Macaulay modules over any standard graded hypersurface ring. 

\begin{cor}\label{cor:MCM} 
Let $T = \kk[x_1,\ldots,x_r]/\<f\>$ for any homogeneous $f$ of degree at least 2, 
and let $\BQMCM(T)$ denote the cone of Betti diagrams of 
maximal Cohen--Macaulay $T$-modules. Then there is an inclusion 
\[
\BQMCM(T) \subseteq \BQ(R),
\] 
where $R = \kk[x]/\<x^{\deg(f)}\>$. 
These cones are equal if $\characteristic(\kk)$ does not divide $\deg(f)$. 
\end{cor}
\begin{proof} 
Let $n$ be the degree of the homogeneous polynomial $f$, 
so that $R = \kk[x]/\<x^n\>$. 
Recall that $T := \kk[x_1,\ldots,x_r]/\<f\>$, and 
let $M$ be a maximal Cohen--Macaulay $T$-module.
To show that $\BQMCM(T) \subseteq \BQ(R)$, 
we find an $R$-module $M'$ with the same Betti diagram. 

We may assume $\kk$ is infinite by taking a flat extension.  
Then we find a sequence of $M$- and $R$-regular linear forms 
$(\ell_1,\ldots,\ell_{r-1})$.  Note that $T/\<\ell_1,\ldots,\ell_{r-1}\> \cong R$. 
Applying~\cite[Corollary 1.2.4]{avramov-notes}, we see that 
$\beta^T(M) = \beta^{T/\<\ell_1,\ldots,\ell_{r-1}\>}(M/\<\ell_1,\ldots,\ell_{r-1}\>)$, 
as desired.

For the second statement, assume that $(\deg f,\characteristic(\kk))=1$. 
Since $\BQ(R) = D$, it is enough to show that for each $\pi_d \in D$, there exists a 
maximal Cohen--Macaulay $T$-module $M_d$ such that $\beta^T(M_d) = \pi_d$.  
If $d = (d_0,\infty,\dots)$, then $\beta^T(T(-d_0))= \pi_d$.  

Now consider the case that $d=(d_0,d_1,d_0+n,\dots)$, 
where without loss of generality $d_0 = 0$ and hence $d_1 < n$. 
In \cite{BHS}, it is shown that there exists a matrix factorization of $f$ that can be 
decomposed into a product of $n$ matrices of linear forms.  
Suppose $A_1 A_2 \cdots A_n$ is such a decomposition.  
If $M := \coker(A_1 \cdots A_{d_1})$ is presented by this matrix of 
$(d_1)$-forms, then it follows that $\beta^T(M) = \pi_d$.
Hence $\BQMCM(T) = \BQ(R)$, as desired. 
\end{proof}

\section{Multiplicity conjectures and decomposition algorithms}
\label{sec:mult and decomp}

In this section, $R$ denotes a standard graded hypersurface rings of the form 
$\kk[x]/\<x^n\>$ for any $n$ or $\kk[x,y]/\<q\>$, where $q$ is any homogeneous quadric. 
We first note that Theorem~\ref{thm:simplicial fan} follows from the proofs of 
Lemma~\ref{lem:FD} and Theorem~\ref{thm:main artinian},  
as they provide the desired simplicial structure. 

This simplicial structure gives rise to 
a greedy decomposition algorithm of Betti diagrams into pure diagrams, as in
\cite[\S1]{eis-schrey1}. The key fact is that, since the cone $\BQ(R)$ is
simplicial, for any module $M$, there is a finite chain of degree
sequences $d_1 < \ldots < d_n$ such that $\beta^R(M)$ is a positive
linear combination of the $\pi_{d_i}$. And as noted in Lemma
\ref{lem:simplicial}, the diagram $\pi_{d_i}$ has a nonzero entry in a
position in which, for all $j > i$, $\pi_{d_j}$ has a zero entry.
We now present a detailed example to illustrate the algorithm.

\begin{example}\label{ex:decomposition}
Let $R=S/\<x^2\>$ and $M=\coker
\begin{pmatrix}
x&xy^2&y^4\\
0&y^3&xy^3
\end{pmatrix}$.  
Then  we have
\[
\beta^R(M)=
\begin{pmatrix}
2&1&1&1&\cdots\\
-&-&-&-&\cdots\\
-&1&-&-&\cdots\\
-&1&1&1&\cdots
 \end{pmatrix}.\]
We decompose $\beta^R(M)$ by first considering the minimal $R$-degree 
sequence that could possibly contribute to $\beta^R(M)$, which is 
$(0,1,2,3,\dots)$.  We then subtract $\frac{1}{2} \pi_{(0,1,2,3,\dots)}$, as this is the largest multiple that can be removed 
while remaining inside $\BQ(R)$. This yields 
\[
\beta^R(M)-\frac{1}{2} \pi_{(0,1,2,3,\dots)}=\begin{pmatrix}
\frac{3}{2}&-&-&-&\cdots\\
-&-&-&-&\cdots\\
-&1&-&-&\cdots\\
-&1&1&1&\cdots
 \end{pmatrix}.
\] 
We next subtract one copy of $\pi_{(0,3,\infty,\infty,\dots)}$, to obtain
\[
\beta^R(M)-\frac{1}{2} \pi_{(0,1,2,3,\dots)}-\pi_{(0,3,\infty,\infty,\dots)}
=\begin{pmatrix}
\frac{1}{2}&-&-&-&\cdots\\
-&-&-&-&\cdots\\
-&-&-&-&\cdots\\
-&1&1&1&\cdots
 \end{pmatrix}.
\]
Note that the remaining Betti diagram equals 
$\frac{1}{2}\pi_{(0,4,5,6,\dots)}$.  
In particular, $\beta^R(M)$ lies in the face corresponding to the chain
\[
(0,1,2,3,\dots)<(0,3,\infty,\infty,\dots)<(0,4,5,6,\dots).
\]
\end{example}

\medskip
The existence of these simplicial structures also gives rise to $R$-analogues of the 
Herzog--Huneke--Srinivasan Multiplicity Conjectures. 
We say that an $R$-degree sequence $d$ is \defi{compatible} with 
a Betti diagram $\beta^R(M)$ if $\beta^R_{i,d_i}(M) \not = 0$ when $d_i < \infty$.

\begin{cor}\label{cor:mult conj}
Let $M$ be an $R$-module generated in a single degree.  
Let $\underline{d}=(\underline{d_0},\underline{d_1},\dots)$ be 
the minimal $R$-degree sequence compatible with $\beta^R(M)$, 
and let $\overline{d}=(\overline{d_0},\overline{d_1},\dots)$ be 
the maximal $R$-degree sequence compatible with $\beta^R(M)$.
\begin{enumerate}
	\item\label{mult:1}  
		We have
		\[
		e(M) \leq \beta^R_0(M)\cdot e(\pi_{\overline{d}}).
		\]
	\item\label{mult:2}  
		If $\overline{d_1}<\infty,$ then 
		\[
		\beta^R_{0}(M)\cdot e(\pi_{\underline{d}})\leq e(M) 
		\leq \beta^R_0(M)\cdot e(\pi_{\overline{d}}), 
		\]
		with equality on either side if and only if 
		$\underline{d}=\overline{d}$.
	\end{enumerate}
\end{cor}
\begin{proof}
Since $M$ is generated in a single degree, we may assume that $\underline{d_0}=0$.  
By Theorem~\ref{thm:simplicial fan}, there is a unique chain 
$\underline{d}=d^0<d^1<\cdots<d^s=\overline{d}$ for which 
\begin{equation}\label{eqn:BSdecomp}
\beta^R(M)=\sum_{i=0}^s a_i\pi_{d^i}. 
\end{equation}
If $\overline{d}=(0,\infty,\infty,\dots)$,  then $M$ has dimension $1$ and
$e(M)=a_{s}e(\pi_{\overline{d}})$. 
Since $a_s\leq \beta^R_{0,0}(M)$, this proves \eqref{mult:1} 
in the case that $\overline{d_1}=\infty$.

We now assume that $\overline{d_1}=\infty$, and 
prove~\eqref{mult:2}, which implies \eqref{mult:1} for this remaining case. 
We first compute the multiplicity of $\pi_d$ for any $R$-degree sequence $d$ of the form 
$d=(0,d_1,d_2,d_3,\dots)$ with $d_1<\infty$. 
We consider separately the cases $\infty=d_2=d_3=\cdots$ 
and $d_i=d_1+i-1$ for all $i\geq 2$.  

We may assume that $\kk$ is infinite by taking a flat extension.  For the first case, 
we may assume after a possible change of coordinates 
that $y$ is a nonzero divisor on $R$.  Then the Betti diagram of $R/\<y^{d_1}\>$ equals 
$\pi_{(0,d_1,\infty,\infty,\dots)}$, and hence
\[
e(\pi_{(0,d_1,\infty,\infty,\dots)})=e(R/\<y^{d_1}\>)=2d_1.
\]
For the remaining case, the Betti diagram of $R/\<y^{d_1},xy^{d_1-1}\>$ equals 
$\pi_{(0,d_1,d_1+1,d_1+2,\dots)}$, and hence
\[
e(\pi_{(0,d_1,d_1+1,d_1+2,\dots)})=e(R/\<y^{d_1},xy^{d_1-1}\>)=2d_1-1.
\]

Note that, since $\overline{d_1}<\infty$, every pure diagram $\pi_{d^i}$ 
arising in the decomposition \eqref{eqn:BSdecomp} satisfies $d^i_0=0$ 
and $d^i_1<\infty$.  Therefore
\[
e(\pi_{d^0})<e(\pi_{d^1})<\cdots<e(\pi_{d^s}).
\]
By convexity, this implies \eqref{mult:2}. 
\end{proof}

\appendix
\section{Convex Geometry}
\label{sec:notation}
In this appendix, we provide background on some convex geometry.   The curious reader may turn to \cite[Chapters 1,2,7]{ziegler} 
for even further details.

Let $V$ be a $\QQ$-vector space. A subset $C \subseteq V$ 
is a \defi{convex cone} if it closed under addition and multiplication by elements of 
$\QQ_{\geq 0}$.  For a subset $B \subseteq V$, $\pos(B)$ denotes the 
\defi{positive hull of $B$}, defined as 
$\pos(B) := \left \{ \sum_{b \in B} a_b b \mid a_b \in \QQ_{\geq 0}
\right \}$, which is clearly a cone. A
\defi{ray} is the $\QQ_{\geq 0}$-span on an element of $V$. 
A ray in a positive hull $\pos(B)$ is an \defi{extremal ray} of $\pos(B)$ 
if it does not lie in $\pos(B\setminus \{b\})$.

We say $C$ is a \defi{$n$-dimensional simplicial cone} if $C
= \pos(B)$ for a set of $n$ linearly independent vectors $B$. An
\defi{$m$-dimensional face} of
$C$ is a subset of the form $\pos(B')$, for $B'$ a subset of
$m$ vectors of $B$. A \defi{facet} of
$C$ is an $(n-1)$-dimensional face.

A \defi{simplicial fan} $\Sigma$ is a collection of simplicial 
cones $\{C_i\}$ such that $C_i \cap C_j$ is a face of both $C_i$ and $C_j$ for all $i,j$.
We refer to $\bigcup_i C_i\subseteq V$ as the \defi{support} of
$\Sigma$.  We say that a subset $\Sigma$ of $V$
has the \defi{structure of a simplicial fan} if $\Sigma$ is the support of some
simplicial fan.

A simplicial fan $\Sigma$ that is a finite union of cones is \defi{$m$-equidimensional} 
if each maximal cone has dimension $m$. 
A \defi{facet} of an equidimensional fan is a facet of any maximal cone, 
and it is a \defi{boundary facet} if it is contained in exactly one maximal cone.  
If $\dim V$ is finite and $\Sigma$ is $(\dim V)$-equidimensional, then each boundary facet $F$
determines a unique, up to scalar, functional $L\colon V \to \QQ$
such that $L$ vanishes along $F$ and is nonnegative on the (unique) maximal cone 
containing $F$;
we refer to the halfspace $\{L \geq 0\}$ as a \defi{boundary halfspace} of the fan.  

Simplicial fan structures that come from posets arise throughout this paper.  
Let $P$ be a poset and assume that there is a map $\Phi\colon P\to V$ such
that $\Phi(p_1), \dots, \Phi(p_s)$ is linearly independent in $V$ for
all chains $p_1 < \ldots < p_s$ in $P$ and such that the union of
simplicial cones
\[
\Sigma(P,\Phi):=\left\{ \pos\left(\{\Phi(p_1), \dots, \Phi(p_s)\} \right) \mid s\in \ZZ_{\geq 0} \text{ and } p_1< \dots < p_s \text{ is a chain in } P \right\}
\]
is a simplicial fan.  (When $P$ is finite, this fan is referred to as a \defi{geometric realization} of $P$.  In our cases, $P$ is the poset of
$R$-degree sequences, and $\Phi$ is the map $d\mapsto \pi_d\in \VV$.)
If $\dim V$ is finite, then maximal cones of $\Sigma(P,\Phi)$ are in bijection with maximal
chains in $P$, and submaximal chains in $P$ are in bijection with facets of $\Sigma(P,\Phi)$.

\begin{lemma}\label{lem:boundaryD} 
Let $V$ be an $m$-dimensional $\QQ$-vector space, $P$ be a finite poset, 
$\Phi\colon P\to V$ as above, and $\Sigma(P,\Phi)$ be an $m$-equidimensional simplicial fan.  Then there is a bijective map:
\begin{align*}
\left\{\text{\parbox[h]{6cm}{\begin{center}submaximal chains of $P$ that\\
lie in a unique maximal chain of $P$
\end{center}}} \right\} 
& \longrightarrow \left\{
\text{\parbox[h]{2.6cm}{\begin{center}
boundary facets 
of $\Sigma(P,\Phi)$
\end{center}}} \right\}
 \end{align*}
 that is given by sending the submaximal chain $p_1<\dots<p_{m-1}$ to $\pos(\{\Phi(p_1),\dots, \Phi(p_m)\})$.  
 
 In addition, since $p_1<\dots<p_{m-1}$ lies in a unique maximal cone, there is a unique $q\in P$ which extends this to a
 maximal chain.  The boundary halfspace determined by this submaximal chain is the halfspace 
 $\{L\geq 0\}$, where $L(\Phi(p_i))=0$ and $L(\Phi(q))>0$. 
Though more than one submaximal chain may determine the same boundary halfspace, each boundary halfspace corresponds to at least one such chain. \hfill \qed
 \end{lemma}

\begin{example}\label{ex:fig1}
Let $P$ be the poset from Figure~\ref{fig:bounded poset}. 
We continue with the notation of the proof of Lemma~\ref{lem:FD}, letting ${D}_1'$ be the simplicial fan on $P$.  
Since $P$ has $12$ maximal chains, it follows that ${D}_1'$ is the union of $12$ simplicial cones (of dimension $9$).
Consider the maximal chain corresponding to the lower left boundary of Figure~\ref{fig:bounded poset}:
there are $7$ submaximal chains that uniquely extend to this maximal chain.  More precisely, there are respectively $0,2,2,0,1,1,0,1$ such submaximal chains
of type \eqref{type1}--\eqref{type01}.
\end{example}

Although simplicial fans are not necessarily convex, we can always construct a convex cone from a simplicial fan. 

\begin{lemma}\label{lem:boundaryD2}
Let $V$ be an $m$-dimensional $\QQ$-vector space, and let $\Sigma$ be an $m$-equidimensional simplicial fan.
Let $\{ \{L_k\geq 0\}\}$ be the set of boundary halfspaces of $\Sigma$.  The convex cone $\bigcap_k \{L_k\geq 0\}$ is a subset of the support of $\Sigma$.
\end{lemma}
\begin{proof}
The arguments in the proof of Theorem~2.15 of~\cite{ziegler} show that 
$\bigcap_k \{L_k\geq 0\}$ is the largest convex cone contained in the support of $\Sigma$. 
\end{proof}


\end{document}